\documentclass[10pt]{amsart}

\usepackage{amssymb}
\usepackage{amsfonts}
\usepackage{amsmath}
\usepackage[active]{srcltx}

\newtheorem{definition}{Definition}[section]
\newtheorem{lem}{Lemma}[section]
\newtheorem{theorem}{Theorem}[section]
\newtheorem{prop}{Proposition}[section]

\newtheorem{rem}{Remark}[section]

\newtheorem{conj}{Conjecture}
\newcommand\bi{{\bf i}}

\newcommand{\R}{\mathbb{R}}

\newcommand{\N}{\mathbb{N}}
\newcommand{\Z}{\mathbb{Z}}

\newcommand{\dimh}{\mbox{dim}_{\mathcal{H}}}
\newcommand{\dimbsup}{\overline{\mbox{dim}_{\mathcal{B}}}}

\everymath{\displaystyle}

\begin{document}

\title{Multifractal formalism for typical probability measures on self-similar sets }
\author{Moez Ben Abid$^{\dag}$}

\maketitle

\begin{abstract}
In this work, we investigate the H\"older spectrum of typical measure (in the Baire category sense) in a general compact set and  we compute the multifractal spectrum of a typical measures supported by a self-similar set. Such mesures verify the multifractal formalism.



\medskip
\noindent$^{\dag}$ \'{E}cole Sup\'{e}rieure des Sciences et de Technologie de Hammam Sousse (Tunisia).
\end{abstract}

\section{Introduction and the main result}
Let $K$ be a compact set of $\R^{d}$ endowed with the metric induced by any norm on $\R^{d}$.

The local H\"older exponent of a positive measure $\mu$ on $K$ at $x\in K$, $h_{\mu}(x)$, is defined by
\[h_{\mu}(x)=\liminf_{r\to0}\frac{\log\mu(B(x,r))}{\log r}\]
where $B(x,r)$ is the ball of center $x$ and radius $r$.
The purpose of the multifractal analysis of a  measure $\mu$ is to investigate the singularity spectrum $d_{\mu}$ of $\mu$, that is the map
\[d_{\mu}:h\geq0\mapsto \dimh(E_{\mu}(h)) \]
where $E_{\mu}(h)=\left\{x\in K:\; h_{\mu}(x)=h\right\}$ and $\dimh$ is the Hausdorff dimension.

Generally it is very difficult to obtain the singularity spectrum directly from the definition of the Hausdorff dimension. To avoid this difficulty, the multifractal formalism provide a formula which link the singularity spectrum to the Legendre transform of mapping defined by averaged quantities of the measure, precisely to the Legendre transform of the  $L^{q}$ spectrum defined as follows.
If $j$ is an integer greater than $1$ let $\mathcal{G}_{j}$ be the partition of $\R^{d}$ into dyadic boxes: $\mathcal{G}_{j}$ is the set of all cubes
\[I_{j,\mathbf{k}}=\prod_{i=1}^{d}\left[ k_{i}2^{-j},(k_{i}+1)2^{-j}\right[\]
where $\mathbf{k}=(k_{1},k_{2},\cdots,k_{d})\in \Z^{d}$.

The $L^{q}$ spectrum of a measure $\mu\in\mathcal{M}(K)$ is the mapping defined for any $q\in\R$ by
\[\tau_{\mu}(q)=\liminf_{j\to+\infty}-\frac{\log\sum_{Q\in\mathcal{G}_{j},\mu(Q)\neq0 }\mu(Q)^{q}}{j\log2}.\]

A classical result (see for example \cite{Fal1}) assert that for all measure $\mu$ for all $h\geq 0$,
\begin{equation}\label{majSpectrumLegendre}
    d_{\mu}(h)\leq \left(\tau_{\mu}\right)^{*}(h):=\inf_{q\in\R}(qh-\tau_{\mu}(q)).
\end{equation}

A important issue in multifractal analysis is to establish when the upper bound (\ref{majSpectrumLegendre}) turns out to be an equality , when this happens we say that the measure $\mu$ \emph{satisfies the multifractal formalism} at $h$. A lot of work has been achieved for specific measures.

\medskip
In the few last years, a particular interest was allocated to generic results (in the sense of Baire or prevalence) on the space of the probability measures endowed with the weak topology or in some space of functions, see for examples \cite{BucNag}, \cite{BucSeu1}, \cite{BucSeu2}, \cite{FraJaf}, \cite{FraJafKah}, \cite{Ols1}, \cite{Ols2}.

\medskip
We denote by $\mathcal{M}(K)$ the space of probability measures on $K$ endowed with the weak topology. Recall that the weak topology on $\mathcal{M}(K)$ is induced by the metric $\varrho$ on $\mathcal{M}(K)$ defined as follows. Let $\mbox{Lip}(K)$ denote the family of Lipschitz functions $f:K\rightarrow\R$ with $\left|f\right|:=\sup_{x\in K}\left|f(x)\right|\leq 1$ and $Lip(f)\leq 1$ where $Lip(f)$ denotes the Lipschitz constant of $f$. If $\mu$ and $\nu$ belong to $\mathcal{M}(K)$ we set

\[\varrho(\mu,\nu)=\sup\left\{\left|\int f d\mu-\int f d\nu\right|: f\in \mbox{Lip}(K)\right\}.\]
then the space $\mathcal{M}(K)$ is complete and separable.

\medskip
In \cite{BucSeu1}, the authors determined the multifractal spectrum for typical measures $\mu$ (in the Baire sens) in $[0,1]^{d}$ and they showed that such measures satisfy the multifractal formalism. They also made the following conjecture

\begin{conj}
For any compact set $K\subset\R^{d}$, there exists a constant $0<D<d$ such that typical measures $\mu$ (in the Baire sens) in $\mathcal{M}(K)$ satisfy: for any $h\in[0,D]$, $d_{\mu}(h)=h$, and if $h>D$, $E_{\mu}(h)=\emptyset$.
\end{conj}
Whether $D$ should be the Hausdorff dimension of $K$ or the lower box dimension of $K$ (or another dimension).

\medskip
In this paper, we give a positive answer to this conjecture in the special case where $K$ is a self-similar set satisfying the open set condition.

\medskip
Let $X$ be a complete metric space. We say that a set $A$ of $X$ is a $G_{\delta}$ set if it can be written as a countably intersection  of dense open sets. We say that a property is typical in $X$ if it holds on residual set, i.e. a set with complement of first Baire category. By the Baire theorem  any $G_{\delta}$ set is dense.

\medskip
Our main result is the following

\begin{theorem}\label{maintheorem}
Let $K$ be a self-similar set satisfying the open set condition. Let $s$ be the Hausdorff dimension of $K$. Then, there exists a $G_{\delta}$ set $\Omega$ of $\mathcal{M}(K)$ such that for all $\mu\in \Omega$,
\begin{itemize}
\item for all $h>s$, $E_{\mu}(h)=\emptyset$
\item for all $h\in[0,s]$, $d_{\mu}(h)=h$.
\end{itemize}
In particular, for every $q\in[0,1]$, $\tau_{\mu}(q)=s(q-1)$, and $\mu$ satisfies the multifractal formalism at every $h\in[0,s]$, i.e. $d_{\mu}(h)=\left(\tau_{\mu}\right)^{*}(h)$.
\end{theorem}

Our paper is organized as follows: in the second section we show, for any  compact $K$, that for typical measure $\mu$ in $\mathcal{M}(K)$, for $h>s$, $E_{\mu}(h)=\emptyset$ where $s$ is the upper box counting dimension, this can be en particular applied to self-similar sets. In the third section we recall some properties of self-similar sets that will be useful for us. Then, using the same approach as \cite{BucSeu1} with suitable modifications we prove the Theorem \ref{maintheorem}.

\section{Results valid on compact $K$}
Let $0\leq s<+\infty$, $\lambda$ a borelian measure on $K$ and $a\in K$. We define the lower $s-$densities of $K$ at $a$ with respect to $\lambda$ by
\[\Theta^{s}_{*}(K,a,\lambda)=\liminf_{r\downarrow 0}(2r)^{-s}\lambda\left(K\cap B(a,r)\right).\]

\medskip
In this section we will prove the following theorems.

\begin{theorem}
Let $K$ be a closed set of  $\R^{d}$.
\begin{enumerate}
\item Let $a\in K$. Then, there exists a  $G_{\delta}$ set $\Omega(a)$ of $\mathcal{M}(K)$ such that for all $\mu\in \Omega(a)$, $h_{\mu}(a)=0$.
\item Let $s\in]0,d]$ and $A\subset K$. Assume that there exists $\lambda$ a finite Borel measure on $K$ such that for all $a\in A$
\[\Theta^{s}_{*}(K,a,\lambda)>0.\]

Then, there exists a  $G_{\delta}$ set $\Omega$ of $\mathcal{M}(K)$ such that  for all $\mu\in \Omega$, for all $x\in A$, $h_{\mu}(x)\leq s$. That is, for all $\mu\in \Omega$, for all $h>s$, $E_{\mu}(h)\cap A=\emptyset$.
\end{enumerate}
\label{th1}
\end{theorem}

\begin{rem}
Let $a\in K$. For all $h>0$ the set $\Lambda_{h}(a)=\left\{\mu\in \mathcal{M}(K);\;h_{\mu}(a)=h\right\}$ is of empty interior. Indeed, if not then using the dense $G_{\delta}$ set $\Omega(a)$ we get $\Lambda_{h}(a)\cap \Omega(a)\neq \emptyset$ which is impossible.
\end{rem}

\medskip
Let $E$ a non-empty bounded subset or $\R^{d}$ let $N_{r}(E)$ be the largest number of disjoint balls of radius $r$ with centers in $E$. The upper box-counting dimension of $E$ is defined as
\[\overline{\mbox{dim}_{B}}E=\limsup_{r\to0}\frac{\log N_{r}(E)}{-\log r}.\]

Already we can prove the following result.
\begin{theorem}
Let $K$ be a closed set of $\R^{d}$ let $s=\overline{\mbox{dim}_{B}}K$. Then, there exists a  $G_{\delta}$ set $\Omega$ of $\mathcal{M}(K)$ such that for all $\mu\in \Omega $, for all $x\in K$, $h_{\mu}(x)\leq s$. That is, for all $\mu\in \Omega$, for all $h>s$, $E_{\mu}(h)=\emptyset$.
\label{th2}
\end{theorem}

\subsection{Proof of Theorem \ref{th1}}
In the sequel, we will always denote by $B(x,r)$ (resp. $\overline{B}(x,r)$) the open (resp. closed) ball of center $x\in X$ and radius $r$, where $X$ any metric space.

\medskip
1)
Let $a\in K$ and $s>0$. Let $(\nu_{n})$ be a dense sequence in $\mathcal{M}(K)$. Let $(d_{n})_{n}$ be a  decreasing sequence to $0$.

Let $\theta>\frac{2}{s}$. Put $\beta_{n}=\frac{1}{\log\left|\log d_{n}\right|}$.  We consider the following sequences $\alpha_{n}=d_{n}^{\beta_{n}}$, $r_{n}=d_{n}^{\theta s}$, $c_{n}=d_{n}^{\frac{\theta}{2}s}$.
Remark that all the sequences are decreasing to $0$.

\medskip
Denote by
\[\mu_{n}=\alpha_{n}\delta_{a}+(1-\alpha_{n})\nu_{n}\]

where $\delta_{a}$ is the Dirac mass at $a$. Since $\rho(\mu_{n},\nu_{n})\leq 2\alpha_{n}\underset{n\to+\infty}{\rightarrow}0$, the sequence $(\mu_{n})_{n}$ is dense in $\mathcal{M}(K)$.

Now put

\[\Omega_{N}(a)=\bigcup_{k\geq N}B(\mu_{k},r_{k})\quad\mbox{and}\quad\Omega(a)=\bigcap_{N=1}^{+\infty}\Omega_{N}.\]
$\Omega(a)$ is a  $G_{\delta}$ set in $\mathcal{M}(K)$ since for all $N$,  $\Omega_{N}(a)$ is a dense open set.

\medskip
Let $\mu\in \Omega$. There exists an increasing sequence $(m_{n})$ of integers such that for all $n$,
\[\rho(\mu,\mu_{m_{n}})\leq r_{m_{n}}.\]

 Since $0<c_{n}\leq d_{n}$, we can construct a Lipschitz function $f_{n}$ which satisfies $0\leq f_{n}(y)\leq c_{n}$ for all $y$ and $f_{n}(y)=c_{n}$ for all $y\in B(a,\frac{d_{n}}{2})$ and $f_{n}(y)=0$ for all $y\notin B(a,d_{n})$ and such that $Lip(f)\leq1$. (For example we can consider the restriction to $K$ of the function $f:\R^{d}\rightarrow \R$, $f(y)=c_{n}$ if $\left\|y-a\right\|\leq \frac{d_{n}}{2}$; $f(y)=-2\frac{c_{n}}{d_{n}}\left\|y-a\right\|+2c_{n}$ if $\frac{d_{n}}{2}<\left\|y-a\right\|\leq d_{n}$; $f(y)=0$ if $\left\|y-a\right\|>d_{n}$).

We have,
\begin{equation}
\int f_{m_{n}}d\mu\leq c_{m_{n}}\mu(B(a,d_{m_{n}})).
\label{eq1}
\end{equation}

In other part, using the property of the function $f$ we get for all $n$

\begin{eqnarray}
\int f_{m_{n}}d\mu_{m_{n}}&\geq &\alpha_{m_{n}}\int f_{m_{n}}d\delta_{a} \nonumber\\
&\geq&\alpha_{m_{n}}\int_{B(a,\frac{d_{m_{n}}}{2})} f_{m_{n}}d\delta_{a} \nonumber \\
&=&\alpha_{m_{n}}c_{m_{n}}.\label{eq2}
\end{eqnarray}

We have $\rho(\mu,\mu_{m_{n}})\leq r_{m_{n}}$, thus using (\ref{eq1}) and (\ref{eq2}), for all $n$

\begin{eqnarray}
c_{m_{n}}\mu(B(a,d_{m_{n}}))&\geq& \int f_{m_{n}}d\mu\nonumber\\
&\geq& \int f_{m_{n}}d\mu_{m_{n}}-\rho(\mu,\mu_{m_{n}})\nonumber\\
&\geq& \alpha_{m_{n}}c_{m_{n}}-r_{m_{n}}\label{eq3}
\end{eqnarray}
then by (\ref{eq3})
\begin{eqnarray*}
\mu(B(a,d_{m_{n}}))&\geq& \alpha_{m_{n}}-\frac{r_{m_{n}}}{c_{m_{n}}}\\
&=&d_{m_{n}}^{\beta_{m_{n}}}-d_{m_{n}}^{\frac{\theta}{2}s}\\
&=&d_{m_{n}}^{\beta_{m_{n}}}(1-d_{m_{n}}^{\frac{\theta}{2}s-\beta_{m_{n}}})
\end{eqnarray*}

Then for $n$ sufficiently large we get
\begin{equation*}
   \mu(B(a,d_{m_{n}}))\geq \frac{1}{2}d_{m_{n}}^{\beta_{m_{n}}}.
\end{equation*}

Finally

\[h_{\mu}(a)=\liminf_{r\to0}\frac{\log\mu(B(a,r))}{\log r}\leq \liminf_{n\to+\infty}\frac{\log\mu(B(a,d_{m_{n}}))}{\log d_{m_{n}}}\leq \lim_{n\to+\infty}\beta_{m_{n}}=0.\]
 This finish the proof of the first point.

\medskip
\noindent 2) Now we must construct a dense $G_{\delta}$ set of $\mathcal{M}(K)$ which is independent of $a\in A$.

Let $(\nu_{n})$ be a dense sequence in $\mathcal{M}(K)$. Let $\nu=\frac{1}{\gamma(K)}\gamma\llcorner_{K}$, (remark that since $\Theta^{s}_{*}(K,a,\lambda)>0$ for some point $a$ then $\lambda(K)>0$).

Let $(\alpha_{n})_{n}$ be a sequence decreasing to $0$. Denote by
 \[\mu_{n}=\alpha_{n}\nu+(1-\alpha_{n})\nu_{n}.\]
 Since $\rho(\mu_{n},\nu_{n})\leq 2\alpha_{n}\underset{n\to+\infty}{\rightarrow}0$, the sequence $(\mu_{n})_{n}$ is dense in $\mathcal{M}(K)$.

Let $\theta>1+\frac{2}{s}$. We consider the following sequences $d_{n}=\exp\left(-\frac{1}{\alpha_{n}}\right)$, $r_{n}=d_{n}^{\theta s}$ and $c_{n}=d_{n}^{\frac{\theta-1}{2}s}$. All the defined sequences are decreasing to $0$. Remark that $\alpha_{n}=d_{n}^{\beta_{n}}$ with $\lim_{n\to\infty}\beta_{n}=0$.

\medskip
Now we set
\[\Omega_{N}=\bigcup_{k\geq N}B(\mu_{k},r_{k})\quad\mbox{and}\quad\Omega=\bigcap_{N=1}^{+\infty}\Omega_{N}.\]
$\Omega$ is a $G_{\delta}$ set in $\mathcal{M}(K)$ since for all $N$,  $\Omega_{N}$ is a dense open set.

\medskip
Let $\mu\in \Omega$. There exists an increasing sequence $(m_{n})$ of integers such that for all $n$,
\[\rho(\mu,\mu_{m_{n}})\leq r_{m_{n}}.\]

Let $a\in A$.  By our hypothesis, there exist $c>0$ and $v>0$,  such that
\begin{equation}\label{eqPropertyP2}
\mbox{if  } 0<r<v,\quad  \nu(B(a,r))\geq cr^{s}.
\end{equation}
Let $f_{n}$ be the Lipschitz function as constructed in the first point of the theorem associated to the newer sequences $(c_{n})$ and $(d_{n})$.

We get for all $n$,
\begin{equation}
\int f_{m_{n}}d\mu\leq c_{m_{n}}\mu(B(a,d_{m_{n}})).
\label{eq1secondpoint}
\end{equation}

In other part, using the property of the function $f$ we get for all $n$ such that $d_{n}\leq v$,

\begin{eqnarray}
\int f_{m_{n}}d\mu_{m_{n}}&\geq &\alpha_{m_{n}}\int f_{m_{n}}d\nu \nonumber\\
&\geq&\alpha_{m_{n}}\int_{B(a,\frac{d_{m_{n}}}{2})} f_{m_{n}}d\nu \nonumber \\
&=&\alpha_{m_{n}}c_{m_{n}}\nu\left(B(a,\frac{d_{m_{n}}}{2})\right)\nonumber\\
&\geq&c'\alpha_{m_{n}}c_{m_{n}}d_{m_{n}}^{s}.\label{eq2secondpoint}
\end{eqnarray}

Then by (\ref{eq1secondpoint}) and (\ref{eq2secondpoint}) we get

\begin{eqnarray*}
\mu(B(a,d_{m_{n}}))&\geq& c'\alpha_{m_{n}}d_{m_{n}}^{s}-\frac{r_{m_{n}}}{c_{m_{n}}}\\
&=&c'd_{m_{n}}^{\beta_{m_{n}}+s}-d_{m_{n}}^{\frac{\theta+1}{2}s}\\
&=&d_{m_{n}}^{\beta_{m_{n}}+s}(c'-d_{m_{n}}^{\frac{\theta-1}{2}s-\beta_{m_{n}}})
\end{eqnarray*}

Then for $n$ sufficiently large we get
\begin{equation*}
   \mu(B(a,d_{m_{n}}))\geq \frac{c'}{2}d_{m_{n}}^{\beta_{m_{n}}+s}.
\end{equation*}

Finally

\[h_{\mu}(a)=\liminf_{r\to0}\frac{\log\mu(B(a,r))}{\log r}\leq \liminf_{n\to+\infty}\frac{\log\mu(B(a,d_{m_{n}}))}{\log d_{m_{n}}}\leq \lim_{n\to+\infty}s+\beta_{m_{n}}=s.\]
 This finish the proof of the second point.

\subsection{Proof of Theorem \ref{th2}}
For all $n\in\N$, we put $L_{n}=N_{2^{-n}}(K)$ and let $a_{1,n},\ldots, a_{L_{n},n}$ be $L_{n}$ points of $K$ such that for $i\neq j$
\[B(a_{i,n},2^{-n})\cap B(a_{j,n},2^{-n})=\emptyset.\]
Let $\alpha_{n}=2^{-\sqrt{n}}$. Let $(\nu_{n})$ be a dense sequence in $\mathcal{M}(K)$. We consider the probability measures
\[\Pi_{n}=L_{n}^{-1}\sum_{i=1}^{L_{n}}\delta_{a_{i,n}}\] and
\[\mu_{n}=\alpha_{n}\Pi_{n}+(1-\alpha_{n})\nu_{n}.\]
 Since $\rho(\mu_{n},\nu_{n})\leq 2\alpha_{n}\underset{n\to+\infty}{\rightarrow}0$, the sequence $(\mu_{n})_{n}$ is dense in $\mathcal{M}(K)$.

 Now put $r_{n}=2^{-(s+2)n}$. We set
\[\Omega_{N}=\bigcup_{k\geq N}B(\mu_{k},r_{k})\quad\mbox{and}\quad\Omega=\bigcap_{N=1}^{+\infty}\Omega_{N}.\]
$\Omega$ is a  $G_{\delta}$ set in $\mathcal{M}(K)$ since for all $N$,  $\Omega_{N}$ is a dense open set.

\medskip
Let $\mu\in \Omega$. There exists an increasing sequence $(m_{n})$ of integers such that for all $n$,
\[\rho(\mu,\mu_{m_{n}})\leq r_{m_{n}}.\]

Let $x\in K$. Since $L_{m_{n}}$ is the largest number of disjoint balls of radius $2^{-m_{n}}$ with centers in $K$ then there exists $i\in\left\{1,\ldots,L_{m_{n}}\right\}$ such that

\[B(x,2^{-m_{n}})\cap B(a_{i,m_{n}},2^{-m_{n}})\neq\emptyset.\]
Thus $a_{i,m_{n}}\in B(x,22^{-m_{n}})$.

\medskip
Let $f_{n}\in Lip(K)$ such that for all $y\in B(x,22^{-m_{n}})$,  $f_{n}(y)=2^{-m_{n}}$, for all $y\notin B(x,42^{-m_{n}})$, $f_{n}(y)=0$ and $0\leq f_{n}\leq 2^{-m_{n}}$.

We get for all $n$,
\begin{equation}
\int f_{n}d\mu\leq 2^{-m_{n}}\mu(B(x,42^{-m_{n}})).
\label{eq1th2}
\end{equation}

In other part, using the property of the function $f_{n}$ we get for all $n$,

\begin{eqnarray}
\int f_{n}d\mu_{m_{n}}&\geq &\alpha_{m_{n}}\int f_{n}d\Pi_{m_{n}} \nonumber\\
&\geq&\alpha_{m_{n}}\int_{B(x,22^{-m_{n}})} f_{n}d\Pi_{m_{n}}\nonumber\\
&\geq&\alpha_{m_{n}}L_{m_{n}}^{-1}\int_{B(x,22^{-m_{n}})} f_{n}d\delta_{a_{i,m_{n}}} \nonumber \\
&=&\alpha_{m_{n}}2^{-m_{n}}L_{m_{n}}^{-1}.\nonumber
\end{eqnarray}

\medskip
Let $t$ such that $\overline{\mbox{dim}_{B}}K=s<t<s+1$. Then there exists $v>0$ such that for all $0<r<v$,
\[N_{r}(K)<r^{-t}.\]
Thus, for $n$ sufficiently large such that $2^{-m_{n}}<v$ we get $L_{m_{n}}^{-1}>2^{-tm_{n}}$. Then
\begin{equation}
\int f_{n}d\mu_{m_{n}}\geq \alpha_{m_{n}}2^{-m_{n}}2^{-tm_{n}}.\label{eq2th2}
\end{equation}

We have $\rho(\mu,\mu_{m_{n}})\leq r_{m_{n}}$, thus using (\ref{eq1th2}) and (\ref{eq2th2}), we get for $n$ sufficiently large

\begin{eqnarray}
2^{-m_{n}}\mu(B(x,42^{-m_{n}}))&\geq& \int f_{n}d\mu\nonumber\\
&\geq& \int f_{n}d\mu_{m_{n}}-\rho(\mu,\mu_{m_{n}})\nonumber\\
&\geq& \alpha_{m_{n}}2^{-m_{n}}2^{-tm_{n}}-r_{m_{n}}\label{eq3th2}
\end{eqnarray}
then by (\ref{eq3th2})
\begin{eqnarray*}
\mu(B(x,42^{-m_{n}}))&\geq& \alpha_{m_{n}}2^{-tm_{n}}-2^{m_{n}}r_{m_{n}}\\
&=&2^{-tm_{n}(1+\frac{1}{t\sqrt{m_{n}}})}-2^{-(s+1)m_{n}}\\
&=&2^{-tm_{n}(1+\frac{1}{t\sqrt{m_{n}}})}(1-2^{(t-(s+1))m_{n}+\sqrt{m_{n}}})
\end{eqnarray*}
$\lim_{n\to+\infty}(t-(s+1))m_{n}+\sqrt{m_{n}}=-\infty$, then for $n$ sufficiently large we get
\begin{equation*}
   \mu(B(x,42^{-m_{n}}))\geq \frac{1}{2}2^{-tm_{n}(1+\frac{1}{t\sqrt{m_{n}}})}.
\end{equation*}

Finally
\begin{eqnarray*}
h_{\mu}(x)&=&\liminf_{r\to0}\frac{\log\mu(B(x,r))}{\log r}\leq \liminf_{n\to+\infty}\frac{\log\mu(B(x,42^{-m_{n}}))}{\log42^{-m_{n}} }\\
&=&\liminf_{n\to+\infty}\frac{\log\mu(B(x,42^{-m_{n}}))}{\log2^{-m_{n}} }\\
&\leq& \lim_{n\to+\infty}t+\frac{1}{\sqrt{m_{n}}}=t.
\end{eqnarray*}
Then, for all $t$ such that $s=\overline{\mbox{dim}_{B}}K<t<s+1$, we have $h_{\mu}(x)\leq t$. Thus $h_{\mu}(x)\leq s$.

We conclude that for all $\mu\in \Omega$, for all $x\in K$, $h_{\mu}(x)\leq s$.

\section{The mulitifractal spectrum of typical measures on self-similar sets }
In this section we focus on the special case where the compact $K$ is a self-similar set. We recall the definition of such set and some related metric facts that will be useful for our purpose.
\subsection{Recalls on self-similar sets}
We refer the reader to \cite{Hut}, \cite{Kus}, \cite{Mos} for more properties of self-similar sets.

\medskip
By $\R^{d}$, $d\geq 1$, we denote the $d-$dimensional Euclidean space, by $\mathcal{B}(x,r)$ the balls $\left\{y:\;\left|x-y\right|<r\right\}$, $x\in\R^{d}$, $r>0$, $\left|\;\right|$ the canonical Euclidean norm.
Let $\mathbf{S}=\left\{S_{1},\cdots,S_{p}\right\}$ be a given set of contractive similitudes, that is
\begin{equation}
\left|S_{i}(x)-S_{i}(y)\right|=\alpha_{i}\left|x-y\right|
\end{equation}

$i=1,\ldots,p$, where we assume that $0<\alpha_{1}\leq\cdots\leq \alpha_{p}<1$. We will call briefly $S_{i}$ an $\alpha_{i}-$similitude.

We use the classical following notations:
For $n\in\mathbb{N}^{*}$ we note $\mathbb{A}_n=\{\mathbf{i}=i_{1}\cdots
i_{n}: \ \forall \, k\in\{1,...,n\}, \ i_{k}\in\{1,...,p\}\}$ the sets of words of length $n$ in the alphabet $\{1,...,p\}$.
$\mathbb{A}^{*}=\bigcup_{n\geq 1} \mathbb{A}_n$. Finally
$\mathbb{A}=\{\mathbf{i}=i_{1}\cdots
i_{k}\cdots: \;i_{k}\in\{1,...,p\}\}$  the set of  infinite words.

Without  confusion we note in bold characters the
elements of $\mathbb{A}$ and $\mathbb{A}^{*}$. For
$\mathbf{i}=i_{1}\cdots i_{n}\in\mathbb{A}_n$, we set
$|\mathbf{i}|=n$ the length of  $\mathbf{i}$.  For
$\mathbf{i} \in\mathbb{A}$ and $n\geq 1$, we note
$\mathbf{i}[n]=i_{1}\cdots i_{n}$.

\medskip
If $\mathbf{i}=i_{1}\cdots i_{n}\in\mathbb{A}_n$,  then $S_{{\bf i}}:= S_{i_{1}}\circ\cdots\circ
S_{i_{n}}$  is a contraction with ratio $\alpha_\bi=\alpha_{i_1}\cdot\cdot\cdot \alpha_{i_n}$. If $T$ is any subset of $\R^{d}$, then $T_{\mathbf{i}}=S_{\mathbf{i}}(T)$ with the convention $T_{\emptyset}=T$ and $S_{\emptyset}=$Id. Particularly, for $\mathbf{i}\in\mathbb{A}_n$  $K_{\mathbf{i}}$ is called an $n-$complex.

\medskip
We say that the self similar $K$ satisfy the open set condition, if there exists an open set $U$ such that for all $i\in\left\{1,\cdots,p\right\}$,

\[S_{i}(U)\subset U \quad \mbox{and}\quad S_{i}(U)\cap S_{j}(U)=\emptyset\; \mbox{ if }i\neq j.\]

\medskip
Let $R>0$. We set
\[I(R)=\left\{\mathbf{i}=i_{1}\cdots i_{n}\in \mathbb{A}^{*}\;:\;\alpha_{\mathbf{i}}\leq R<\alpha_{\mathbf{i}[n-1]} \right\}.\]
Let $s$ the real defined by
\[\sum_{k=1}^{p}\alpha_{k}^{s}=1.\]
Under the open set condition we have  $\dimh(K)=s$ and $0<\mathcal{H}^{s}(K)<+\infty$  where $\mathcal{H}^{s}$ is the $s-$Hausdorff measure (see for example \cite{Fal1}). We set
\[\lambda=\frac{1}{\mathcal{H}(K)}\mathcal{H}^{s}\llcorner K.\]

\medskip
In the sequel we denote by $\sharp A$ the cardinality of the set $A$.

\medskip
We gather the useful properties for us in the following proposition (see \cite{Mos}, \cite{Quef}. Some results are also in the proof of the Theorem 9.3 in \cite{Fal1}).
\begin{prop}
We assume that the open set condition is satisfied with the open set $U$. Let $s=\dimh(K)$.
\begin{enumerate}
\item There exist two constants $c_{1},c_{2}>0$, such that for all $R>0$,
\begin{equation*}\label{cardinalIR}
c_{1}R^{-s}\leq \sharp I(R)\leq c_{2}R^{-s}.
\end{equation*}
\item For all $R>0$,
\[\sum_{\mathbf{i}\in I(R)}\alpha_{\mathbf{i}}=1.\]
\item There exist two constants $c_{1},c_{2}>0$ such that for all $R>0$, for all $x\in K$
\begin{equation*}\label{propLambda1}
c_{1}R^{s}\leq \lambda\left(B(x,R)\right)\leq c_{2}R^{s}.
\end{equation*}
\item Let $R>0$.  
\begin{enumerate}
\item $K=\bigcup_{\mathbf{i}\in I(R)}K_{\mathbf{i}}$.
\item For all $\mathbf{i}, \mathbf{j}\in I(R)$ such that $\mathbf{i}\neq \mathbf{j} $, we have
\begin{equation*}\label{propLambda2}
U_{\mathbf{i}}\cap U_{\mathbf{j}}=\emptyset\quad\mbox{and}\quad \lambda\left(K_{\mathbf{i}}\cap K_{\mathbf{j}}\right)=0.
\end{equation*}
\end{enumerate}
\end{enumerate}
\label{propositionSelfsimilarset}
\end{prop}

\subsection{Proof of Theorem \ref{maintheorem}}
Let $K$ be a self-similar set associated to the system $\mathbf{S}=\left\{S_{1},\cdots,S_{p}\right\}$ of $\alpha_{i}-$simiitudes satisfying the open set condition, where we assume that $0<\alpha_{1}\leq\cdots\leq \alpha_{p}<1$. We adopt the notations of the previous section. Denote by $s=\dimh K$.

\medskip
Since $s=\overline{\mbox{dim}_{B}}K$, then by Theorem \ref{th2}, we know that there exists a $G_{\delta}$ set $\Omega'$ of $\mathcal{M}(K)$ such that for all $\mu\in\Lambda$, for all $x\in K$, $h_{\mu}(x)\leq s$.

To achieve the proof of the Theorem \ref{maintheorem}, we will prove that there exists a $G_{\delta}$ set $\Omega''$ of $\mathcal{M}(K)$ such that for all $\mu\in\Omega''$, for all $h\in]0,s]$, $d_{\mu}(h)=h$. Then to recover $h=0$, we fix any point $x_{0}\in K$ and we consider the $G_{\delta}$ set $\Omega(x_{0})$ associated to $x_{0}$ in Theorem \ref{th1}. We consider finally $\Omega=\Omega'\cap \Omega''\cap \Omega(x_{0})$  which stills a $G_{\delta}$ set of $\mathcal{M}(K)$.

\medskip
\begin{rem}
In our proofs many constants will appear with no importance. To relieve the work, we will sometimes denote the constants by the same letter between consecutive inequalities even if the constants are different.
\end{rem}

\medskip
We adopt the same approach  of \cite{BucSeu1} with suitable modifications.

\medskip
For any $\mathbf{i}\in I(2^{-J_{N}})$ we pick an $x_{\mathbf{i}}\in K_{\mathbf{i}}$. The family of point $\bigcup_{N}\left\{x_{\mathbf{i}}:\; \mathbf{i}\in I(2^{-J_{N}})\right\}$ will be fixed in the rest of the paper.

\medskip
We define the following probability measure
\[\lambda_{n}=\sum_{\mathbf{i}\in I(2^{-J_{n}}) }\alpha_{\mathbf{i}}^{s}\delta_{x_{\mathbf{i}}}\]
where $\delta_{x_{\mathbf{i}}}$ is the Dirac mass at the point $x_{\mathbf{i}}$.  $\lambda_{n}$ is probability measure since $\sum_{\mathbf{i}\in I(2^{-J_{n}}) }\alpha_{\mathbf{i}}^{s}=1$ (see Proposition \ref{propositionSelfsimilarset}), and is supported by $K$.

\medskip
Let $\beta_{n}=\frac{J_{n}}{n}$ and we denote by
\[\mu_{n}=\beta_{n}\lambda_{n}+(1-\beta_{n})\nu_{n}\]
the sequence $(\mu_{n})_{n}$ is dense sequence in $\mathcal{M}(K)$ since $\varrho(\mu_{n},\nu_{n})\leq 2\beta_{n}$.

\begin{definition}
 Let $n\in\N^{*}$. We introduce
 \[\Omega_{n}=\bigcup_{k\geq n}B(\mu_{k},2^{-sJ_{k}^{2}}) \quad\mbox{and}\quad \Omega''=\bigcap_{n\geq1}\Omega_{n}\]
\end{definition}
$\Omega''$ is a  $G_{\delta}$ set of $\mathcal{M}(K)$.

\medskip
Let $\mu\in\Omega''$ be fixed. There exists a sequence $(J_{N_{p}})_{p\geq1}$ such that for all $p$, \[\varrho(\mu,\mu_{N_{p}})<2^{-sJ_{N_{p}}^{2}}.\]

\begin{definition}
Let $\theta\geq1$. Let us introduce the set of points
\[\Lambda_{\theta,p}=\bigcup_{\mathbf{i}\in I(2^{-J_{N_{p}}})}\overline{B}(x_{\mathbf{i}},2^{-\theta J_{N_{p}}})\cap K.\]
and then let us define
\[\Lambda_{\theta}=\bigcap_{P\geq 1}\bigcup_{p\geq P}\Lambda_{\theta,p}.\]

\end{definition}

\begin{lem}
Let $\epsilon>0$. There exist $p_{\epsilon}$ and $c>0$, such that for all $p\geq p_{\epsilon}$ and for all $x\in\Lambda_{\theta,p}$,
\begin{equation}{\label{majholdmu}}
\mu(B(x,2 2^{-\theta J_{N_{p}}}))\geq c 2^{-s(1+\epsilon)J_{N_{p}}}.
\end{equation}

\end{lem}
\begin{proof}
Let $\epsilon>0$  and  $x\in \Lambda_{\theta,p}$. Then there exists $\mathbf{i}\in I(2^{-J_{N_{p}}})$ such that $x_{\mathbf{i}}\in \overline{B}(x,2^{-\theta J_{N_{p}}})$. Thus
\[\mu_{N_{p}}(\overline{B}(x,2^{-\theta J_{N_{p}}}))\geq \beta_{N_{p}}\alpha_{\mathbf{i}}^{s}\geq c_{1}\beta_{N_{p}}2^{-sJ_{N_{p}}}=c_{1}2^{-s(1+\frac{1}{sN_{p}})J_{N_{p}}}\geq c_{1}2^{-s(1+\epsilon)J_{N_{p}}}\]
for $p$ so large that $\frac{1}{sN_{p}}< \epsilon$.

Let $f_{\theta,p}$ be a lipschitz function on $K$ with $f\in \mbox{Lip}(K)$ such that for all $z\in\overline{B}(x,2^{-\theta J_{N_{p}}})$, $f_{\theta,p}(z)=2^{-\theta J_{N_{p}}}$, for all $z\notin \overline{B}(x,22^{-\theta J_{N_{p}}})$,  $f_{\theta,p}(z)=0$, and $0\leq f_{\theta,p}\leq 2^{-\theta J_{N_{p}}}$. By construction,
\[\int f_{\theta,p}d\mu\leq 2^{-\theta J_{N_{p}}}\mu(B(x,22^{-\theta J_{N_{p}}}))\]
and

\[\int f_{\theta,p}d\mu_{N_{p}}\geq c_{1}2^{-\theta J_{N_{p}}}2^{-s(1+\epsilon)J_{N_{p}}}\]
thus
\begin{eqnarray*}
2^{-\theta J_{N_{p}}}\mu(B(x,22^{-\theta J_{N_{p}}}))&\geq&\int f_{\theta,p}d\mu\\
&\geq&\int f_{\theta,p}d\mu_{N_{p}}-\varrho(\mu,\mu_{N_{p}})\\
&\geq& c_{1}2^{-\theta J_{N_{p}}}2^{-s(1+\epsilon)J_{N_{p}}}-2^{-sJ_{N_{p}}^{2}}
\end{eqnarray*}
when $p$ is sufficiently large
\[2^{-sJ_{N_{p}}^{2}}\leq \frac{1}{2}c_{1}2^{-\theta J_{N_{p}}}2^{-s(1+\epsilon)J_{N_{p}}}\]
thus there exists $p_{\epsilon}$ such that for all $p\geq p_{\epsilon}$,
\[\mu(B(x,22^{-\theta J_{N_{p}}}))\geq \frac{1}{2}c_{1}2^{-s(1+\epsilon)J_{N_{p}}}.\]

\end{proof}
\begin{prop}\label{propDimHausdorffLambdaTheta}
Let $\theta\geq 1$ and $x\in \Lambda_{\theta}$. Then $h_{\mu}(x)\leq \frac{s}{\theta}$.
\end{prop}
\begin{proof}
If $x\in \Lambda_{\theta}$, then (\ref{majholdmu}) is satisfied for infinite number of integer $p$. Hence, for all $\epsilon>0$, there is a sequence of infinite real numbers $(r_{p})$ decreasing to $0$ such that for all $p$
\[\mu(B(x,2 r_{p}))\geq c r_{p}^{\frac{s}{\theta}(1+\epsilon)}\]
this implies that $h_{\mu}(x)\leq\frac{s}{\theta}(1+\epsilon) $ for all $\epsilon>0$, the result follows.
\end{proof}

\begin{prop}
For all $\theta\geq 1$, $\dimh\Lambda_{\theta}\leq\frac{s}{\theta}$.
\end{prop}
\begin{proof}
The result is obvious when $\theta=1$, since $\Lambda\subset K$ and $\mbox{dim}_{H}(K)=s$.

Let $\theta>1$ and $t>\frac{s}{\theta}$. For all $P\geq 1$, $\Lambda_{\theta}$ is covered by $\bigcup_{p\geq P}\Lambda_{\theta,p}$.
Hence, for any $\delta>0$ and using the fact that $\sharp I(R)\leq cR^{-s}$ (see \cite{Quef}) we obtain
\begin{eqnarray*}
\mathcal{H}_{\delta}^{t}(\Lambda_{\theta})&\leq& \mathcal{H}_{\delta}^{t}\left(\bigcup_{p\geq P}\Lambda_{\theta,p}\right) \\
&\leq& \sum_{p\geq P}\sum_{\mathbf{i\in I(2^{-J_{N_{p}}})}}\left|\overline{B}(x_{\mathbf{i}},2^{-\theta J_{N_{p}}})\right|^{t}\\
&\leq& c\sum_{p\geq P}2^{-t\theta J_{N_{p}}}\sharp I(2^{-J_{N_{p}}})\\
&\leq&c \sum_{p\geq P}2^{-t\theta J_{N_{p}}}2^{s J_{N_{p}}}
\end{eqnarray*}
since $t>\frac{s}{\theta}$, this series is convergent. Hence, $\mathcal{H}_{\delta}^{t}(\Lambda_{\theta})\leq c\lim_{P\to\infty}\sum_{p\geq P}2^{-t\theta J_{N_{p}}}2^{s J_{N_{p}}}=0$. Thus, $\mathcal{H}^{t}(\Lambda_{\theta})=\lim_{\delta\to0}\mathcal{H}_{\delta}^{t}(\Lambda_{\theta})=0$. This implies that $\mbox{dim}_{H}(\Lambda_{\theta})\leq t$ for all $t>\frac{s}{\theta}$, then we conclude.
\end{proof}

\bigskip
Let $m$ be any Borel measure on $K$. The Hausdorff dimension of $m$ is defined by
\[\dimh m=\inf\left\{\dimh E\;:\; E\subset K,\;m(E)>0\right\}.\]
When $m(E)>0$ then $\dimh E\geq \dimh m$. In other words,
\begin{equation}\label{propertyDimMesure}
\mbox{if }\dimh E< \dimh m,\;\mbox{then }m(E)=0.
\end{equation}

\bigskip
As in  \cite{BucSeu1} we have the following result
\begin{theorem}\label{thMtheta}
For every $\theta\geq1$, there is a measure $m_{\theta}$  supported in $\Lambda_{\theta}$, a constant $C>0$ and a positive sequence $(\eta_{p})$ decreasing to $0$ such that for every Borel set $B$,
\begin{equation}
\mbox{if }\left|B\right|\leq \eta_{p},\quad m_{\theta}(B)\leq C\left|B\right|^{\frac{s}{\theta}-\frac{2}{p-1}}.
\end{equation}
In particular, $\dimh m_{\theta}\geq \frac{s}{\theta}$.
\end{theorem}
\begin{proof}
We will construct a suitable Cantor set $\mathcal{C}_{\theta}$ included in $\Lambda_{\theta}$ and a measure $m_{\theta}$ supported on $\mathcal{C}_{\theta}$ with monofractal behaviour.

\medskip
We suppose that the sequence $(J_{N_{p}})$ is sufficiently rapidly decreasing. Precisely, assume that
\begin{equation}\label{eqDecreJnp}
J_{N_{p+1}}>\max((p+1)\theta J_{N_{p}},e^{J_{N_{p}}}).
\end{equation}

\medskip
The following lemma will be useful for us to control the cardinality of some sets of balls.
\begin{lem}\label{lemNombreIntersecBoule}
Let $R>0$, $\mathbf{i}\in I(R)$. For every $c>0$,  there exists $M\in\N^{*}$ independent of $R$ such that
\begin{equation}
\sharp \left\{\mathbf{j}\in I(R);\;\overline{B}(x_{\mathbf{i}},cR)\cap\overline{B}(x_{\mathbf{j}},cR)\neq\emptyset\right\}\leq M.
\end{equation}
\end{lem}

\begin{proof}
In the sequel we denote by $\lambda$ the measure associated to the self similar $K$, see Proposition \ref{propositionSelfsimilarset} for its properties.

\medskip
Denote by
\[T_{\mathbf{i}}=\sharp \left\{\mathbf{j}\in I(R);\;\overline{B}(x_{\mathbf{i}},cR)\cap\overline{B}(x_{\mathbf{j}},cR)\neq\emptyset\right\}.\]

Since $\overline{B}(x_{\mathbf{i}},cR)\cap\overline{B}(x_{\mathbf{j}},cR)\neq\emptyset$ and $\left|K_{\mathbf{j}}\right|=\left|K\right| \alpha_{\mathbf{i}}\leq \left|K\right| R$ then there exists a constant $a$ independent of $R$ such that
\[K_{\mathbf{j}}\subset B(x_{\mathbf{i}},a R).\]
Hence
\[T_{\mathbf{i}}\leq \sharp \left\{\mathbf{j}\in I(R):\; K_{\mathbf{j}}\subset B(x_{\mathbf{i}},a R)\cap K \right\}.\]
 It follows that
 \[\lambda\left(\bigcup_{\mathbf{j};\;K_{\mathbf{j}}\subset B(x_{\mathbf{i}},a R)}K_{\mathbf{j}}\right)\leq \lambda\left(B(x_{\mathbf{i}},a R)\cap K\right)\leq cR^{s}.\]

Since for $\mathbf{j}\neq\mathbf{j'}\in I(R)$ $\lambda\left(K_{\mathbf{j}}\cap K_{\mathbf{j'}}\right)=0$, we obtain
\[\sum_{\mathbf{j}\in I(R):\; K_{\mathbf{j}}\subset B(x_{\mathbf{i}},a R)}\lambda(K_{\mathbf{j}})\leq c R^{s}.\]
We know that for all $\mathbf{j}\in I(R)$, $c'R^{s}\leq \lambda(K_{\mathbf{j}})$, which gives
\[c'\sharp \left\{\mathbf{j}\in I(R):\; K_{\mathbf{j}}\subset B(x_{\mathbf{i}},a R)\cap K\right\} R^{s}\leq c R^{s}.\]
Then
\[\sharp \left\{\mathbf{j}\in I(R):\; K_{\mathbf{j}}\subset B(x_{\mathbf{i}},a R)\cap K\right\}\leq \frac{c}{c'}.\]
Since $T_{\mathbf{i}}\leq \sharp \left\{\mathbf{j}\in I(R):\; K_{\mathbf{j}}\subset B(x_{\mathbf{i}},a R)\cap K \right\}$ we get the desired result.
\end{proof}

\medskip
Denote by $\widetilde{F}_{1}$ a set formed by the largest number of disjoint balls $B(x_{\mathbf{i}},2^{- J_{N_{1}}})$, $\mathbf{i}\in I(2^{-J_{N_{1}}})$. We denote by
\[D_{1}=\left\{\mathbf{i}\in I(2^{-J_{N_{1}}}):\;x_{\mathbf{i}}\mbox{ is a center of ball in }\widetilde{F}_{1}\right\}.\]
Then, we set
\[F_{1}=\left\{\overline{B}(x_{\mathbf{i}},2^{-\theta J_{N_{1}}})\cap K:\;\mathbf{i}\in D_{1}\right\}.\]

We denote by $\Delta_{1}=\sharp F_{1}$. Remark that $\sharp F_{1}=\sharp\widetilde{F}_{1}$. Using the Lemma \ref{lemNombreIntersecBoule}, we get
\[\frac{\sharp I(2^{-J_{N_{1}}})}{M}\leq\Delta_{1}\leq \sharp I(2^{-J_{N_{1}}}).\]
Hence
  \[c\frac{2^{sJ_{N_{1}}}}{M}\leq\Delta_{1}\leq c'2^{sJ_{N_{1}}}.\]

We define a probability measure $m_{1}$  by giving the value $m_{1}(V)=\frac{1}{\Delta_{1}}$ for each element $V\in F_{1}$ and then we extend $m_{1}$ to a Borel probability measure on the algebra generated by $F_{1}$, i.e. on $\sigma(V:\;V\in F_{1})$.

\medskip
Assume that we have constructed $F_{1},\ldots, F_{p}$, $p\geq 1$ and a measure $m_{p}$ on the algebra $\sigma(V:\;V\in F_{p})$. Let $V\in F_{p}$. There exists $\mathbf{i}\in I(2^{-J_{N_{p}}})$ such that $V=\overline{B}(x_{\mathbf{i}},2^{-\theta J_{N_{p}}})\cap K$.

\medskip
Let $r_{p}=2^{-\theta J_{N_{p}}}-\left|K\right|2^{-J_{N_{p+1}}}$. We have $\frac{1}{2}2^{-\theta J_{N_{p}}}\leq r_{p}\leq 2^{-\theta J_{N_{p}}}$. Let us consider
\[D_{p,\mathbf{i}}=\left\{\mathbf{j}\in I(2^{-J_{N_{p+1}}}):\; K_{\mathbf{j}}\cap B(x_{\mathbf{i}},r_{p})\neq\emptyset \right\}.\]

\begin{lem}\label{lemCardinalDpi}
There exist two constants $c_{1},c_{2}>0$ such that for all $p$
\begin{equation}
c_{1}2^{-s\theta J_{N_{p}}}2^{sJ_{N_{p+1}}}\leq \sharp D_{p,\mathbf{i}}\leq c_{2}2^{-s\theta J_{N_{p}}}2^{sJ_{N_{p+1}}}.
\end{equation}

\end{lem}
\begin{proof}
Recall that $K=\bigcup_{\mathbf{j}\in I(2^{-J_{N_{p+1}}})}K_{\mathbf{j}}$. Then,
\begin{eqnarray*}
B(x_{\mathbf{i}},r_{p})\cap K&=&\bigcup_{\mathbf{j}\in I(2^{-J_{N_{p+1}}})}B(x_{\mathbf{i}},r_{p})\cap K_{\mathbf{j}}\\
&=&\bigcup_{\mathbf{j}\in D_{p,\mathbf{i}}}B(x_{\mathbf{i}},r_{p})\cap K_{\mathbf{j}}
\end{eqnarray*}
thus
\begin{eqnarray*}
\lambda\left(B(x_{\mathbf{i}},r_{p})\cap K\right)&=&\lambda\left(\bigcup_{\mathbf{j}\in D_{p,\mathbf{i}}}B(x_{\mathbf{i}},r_{p})\cap K_{\mathbf{j}}\right)\\
&\leq& \sum_{\mathbf{j}\in D_{p,\mathbf{i}}}\lambda(K_{\mathbf{j}})\\
&\leq& c \sharp D_{p,\mathbf{i}} 2^{-sJ_{N_{p+1}}}
\end{eqnarray*}
But $\lambda\left(B(x_{\mathbf{i}},r_{p})\cap K\right)\geq c'r_{p}^{s}\geq c''2^{-s\theta J_{N_{p}}}$, hence there exists $c_{1}$ such that
\[c_{1}2^{-s\theta J_{N_{p}}}2^{sJ_{N_{p+1}}}\leq \sharp D_{p,\mathbf{i}}.\]

In the other hand, for $\mathbf{j}\in D_{p,\mathbf{i}}$, $K_{\mathbf{j}}\cap B(x_{\mathbf{i}},r_{p})\neq\emptyset$, this implies that $K_{\mathbf{j}}\subset B(x_{\mathbf{i}},2^{-\theta J_{N_{p}}}) $ (since $\left|K_{\mathbf{j}}\right|\leq \left|K\right|2^{-J_{N_{p+1}}}$). Thus,
\[\bigcup_{\mathbf{j}\in D_{p,\mathbf{i}}} K_{\mathbf{j}}\subset B(x_{\mathbf{i}},2^{-\theta J_{N_{p}}})\cap K\]
hence
\begin{eqnarray*}
\lambda\left(\bigcup_{\mathbf{j}\in D_{p,\mathbf{i}}} K_{\mathbf{j}}\right)&\leq& \lambda\left(B(x_{\mathbf{i}},2^{-\theta J_{N_{p}}})\cap K\right)\\
&\leq& c 2^{-s\theta J_{N_{p}}}
\end{eqnarray*}
But, $\lambda\left(\bigcup_{\mathbf{j}\in D_{p,\mathbf{i}}} K_{\mathbf{j}}\right)=\sum_{\mathbf{j}\in D_{p,\mathbf{i}}}\lambda\left(K_{\mathbf{j}}\right)$ (since $\lambda\left(K_{\mathbf{j}}\cap K_{\mathbf{j'}}\right)=0$ for $\mathbf{j}\neq\mathbf{j'}\in I(2^{-J_{N_{p+1}}})$). Thus

 \[ \sum_{\mathbf{j}\in D_{p,\mathbf{i}}}\lambda\left(K_{\mathbf{j}}\right) \leq c 2^{-s\theta J_{N_{p}}}\]
 but
 \[\sum_{\mathbf{j}\in D_{p,\mathbf{i}}}\lambda\left(K_{\mathbf{j}}\right)\geq c' \sharp D_{p,\mathbf{i}} 2^{-sJ_{N_{p+1}}}\]
 thus, there exists $c_{2}$ such that
 \[\sharp D_{p,\mathbf{i}}\leq c_{2}2^{-s\theta J_{N_{p}}}2^{sJ_{N_{p+1}}}.\]
\end{proof}

\medskip
Let us define the set $\widetilde{F}_{p+1}(V)$ formed by the largest number of balls $B(x_{\mathbf{j}},2^{-\theta J_{N_{p+1}}}):\; \mathbf{j}\in D_{p,\mathbf{i}}$ such that if $\mathbf{j}\neq\mathbf{j}'\in D_{p,\mathbf{i}}$ we have
\[B(x_{\mathbf{j}},2^{-J_{N_{p+1}}})\cap B(x_{\mathbf{j'}},2^{-J_{N_{p+1}}})=\emptyset.\]
Then, we set
\[F_{p+1}(V)=\left\{\overline{U}\cap K:\; U\in \widetilde{F}_{p+1}(V)\right\}.\]
Remark that for all $U\in F_{p+1}(V)$, $U\subset V=\overline{B}(x_{\mathbf{i}},2^{-\theta J_{N_{p}}})\cap K$.

\medskip
We have the following lemma
\begin{lem}\label{lemCardinalFp}
There exist two constants $c'_{1},c'_{2}>0$ such that for all $p$
\begin{equation}
c'_{1}2^{-s\theta J_{N_{p}}}2^{sJ_{N_{p+1}}}\leq \sharp F_{p+1}(V)\leq c'_{2}2^{-s\theta J_{N_{p}}}2^{sJ_{N_{p+1}}}.
\end{equation}
\end{lem}
\begin{proof}
We have $\sharp F_{p+1}(V)=\sharp \widetilde{F}_{p+1}(V)$. Since $\sharp \widetilde{F}_{p+1}(V)\leq \sharp D_{p,\mathbf{i}}$, we get obviously the second inequality.

Using the lemma \ref{lemNombreIntersecBoule} we can pick at least $\frac{c'}{M}\sharp D_{p,\mathbf{i}}$ element of $\widetilde{F}_{p+1}(V)$ such that the balls of radius $2^{-J_{N_{p}}}$ are disjoint. Since $\widetilde{F}_{p+1}(V)$ is of largest cardinality  then we conclude that
\[\frac{c'}{M}\sharp D_{p,\mathbf{i}}\leq \widetilde{F}_{p}(V)\]
and then we get the first inequality by using the lemme \ref{lemCardinalDpi}.
\end{proof}

\medskip
Now we define  $F_{p+1}=\bigcup_{V\in F_{p}}F_{p}(V)$.
We define a probability measure $m_{p+1}$ by giving the mass $m_{p+1}(U)=\frac{m_{p}(V)}{\sharp F_{p+1}(V)}$, where $V$ is the unique element of $F_{p}$ containing $U$. We extend then $m_{p+1}$ to $\sigma(U:\;U\in F_{p+1})$.

Finally we set
\[\mathcal{C}_{\theta}=\bigcap_{p\geq1}\bigcup_{V\in F_{p}}V.\]

\noindent By the Kolmogorov extension theorem, $(m_{p})_{p\geq1}$ converges weakly to a Borel probability measure $m_{\theta}$ supported on $\mathcal{C}_{\theta}$ and such that for every $p\geq 1$, for every $V\in F_{p}$, $m_{\theta}(V)=m_{p}(V)$.

\subsubsection{Hausdorff dimension of $\mathcal{C}_{\theta}$ and $m_{\theta}$} As is \cite{BucSeu1} we first prove that $m_{\theta}$ has an almost monofractal behavior on set belonging to $\bigcup_{p}F_{p}$.

\begin{lem}
When $p$ is sufficiently large, for every $V\in F_{p}$
\begin{equation}\label{eq1Mtheta}
2^{-sJ_{N_{p}}(1+\frac{1}{p})}\leq m_{\theta}(V)\leq 2^{-sJ_{N_{p}}(1-\frac{2}{p})}
\end{equation}
and
\begin{equation}\label{eq2Mtheta}
\left|V\right|^{\frac{s}{\theta}+\frac{1}{\left|\log\left|V\right|\right|}}\leq m_{\theta}(V)\leq \left|V\right|^{\frac{s}{\theta}-\frac{1}{\left|\log\left|V\right|\right|}}.
\end{equation}
\end{lem}
\begin{proof}
Let $V\in F_{p}$. We denote  $\Delta_{p+1}(V)=\sharp F_{p+1}(V)$ (recall that $F_{p+1}(V)$ is the set of element of $F_{p+1}$ included in $V$). For $k\leq p$ denote by $V_{k}$ the unique element in $F_{k}$ containing $V$. By construction of the measure $m_{\theta}$ we obtain
\[m_{\theta}(V)=\left(\prod_{k=1}^{p}\Delta_{k}(V_{k-1})\right)^{-1}.\]
Using  Lemma \ref{lemCardinalFp}, there exist two constants $c'_{1},c'_{2}>0$ such that
\[c'_{1}2^{-s\theta J_{N_{k}}}2^{sJ_{N_{k-1}}}\leq \Delta_{k}(V_{k-1})\leq c'_{2}2^{-s\theta J_{N_{k-1}}}2^{sJ_{N_{k}}}\]

\noindent by (\ref{eqDecreJnp}) and the fact that $2^{-s\theta J_{N_{k-1}}}\leq1$ we get
\[c'_{1}2^{sJ_{N_{k}}(1-\frac{1}{k})}\leq\Delta_{k}(V_{k-1})\leq c'_{2}2^{sJ_{N_{k}}}.\]
Hence
\[(c'_{2})^{-p}\left(\prod_{k=1}^{p}2^{sJ_{N_{k}}}\right)^{-1}\leq m_{\theta}(V)\leq (c'_{1})^{-p}\left(\prod_{k=1}^{p}2^{sJ_{N_{k}}(1-\frac{1}{k})}\right)^{-1}.\]
Recalling that by (\ref{eqDecreJnp}), $J_{N_{k}}>e^{J_{N_{k-1}}}$ for every $k$, thus
\[\lim_{p\to+\infty}\frac{p}{J_{N_{p}}}\left(p\frac{\log c'_{2}}{s\log2}+\sum_{k=1}^{p-1}J_{N_{k}}\right)=0\]
(since $\sum_{k=1}^{p-1}J_{N_{k}}\leq (p-1)J_{N_{p-1}}\leq (J_{N_{p-1}})^{2}$, we get $\frac{p\sum_{k=1}^{p-1}J_{N_{k}}}{J_{N_{p}}}\leq p(J_{N_{p-1}})^{2}e^{-J_{N_{p-1}}}\leq (J_{N_{p-1}})^{3}e^{-J_{N_{p-1}}}\underset{p\to+\infty}{\rightarrow}0$).
Thus there exists $p_{1}$ such that for all $p\geq p_{1}$,
\[(c'_{2})^{-p}\left(\prod_{k=1}^{p-1}2^{sJ_{N_{k}}(1-\frac{1}{k})}\right)^{-1}\geq 2^{-\frac{s}{p}J_{N_{p}}} \]
hence, for all $p\geq p_{1}$,
\[2^{-sJ_{N_{p}}(1+\frac{1}{p})}\leq m_{\theta}(V).\]

As previously, $\lim_{p\to+\infty}\frac{p}{J_{N_{p}}}\left(p\frac{\log c'_{1}}{s\log2}+\sum_{k=1}^{p-1}J_{N_{k}}(1-\frac{1}{k})\right)=0$. Thus there exists $p_{2}$ such that for all $p\geq p_{2}$
\[(c'_{1})^{-p}\left(\prod_{k=1}^{p-1}2^{sJ_{N_{k}}(1-\frac{1}{k})}\right)^{-1}\leq 2^{\frac{s}{p}J_{N_{p}}}\]
hence for all $p\geq p_{2}$
\[m_{\theta}(V)\leq 2^{-sJ_{N_{p}}(1-\frac{2}{p})}.\]

To prove (\ref{eq2Mtheta}), remark that for all $V\in F_{p}$, $\left|V\right|\approx 2^{-\theta J_{N_{p}}}$ (where $\approx$ means that the ratio of the two quantities is bounded from below and above by two positives constants independents of $p$ ). Then, $p=o(\left|\log\left|V\right|\right|)$. Thus (\ref{eq1Mtheta}) yields (\ref{eq2Mtheta}).
\end{proof}
Now we extend $(\ref{eq2Mtheta})$ to all Borel subsets of $K$.
\begin{lem}
There is two positive sequences $(\eta_{p})_{p}$, decreasing to $0$ and a constant $C>0$ such that for any Borel set $B\subset K$ with $\left|B\right|\leq \eta_{p}$ we have
\begin{equation}\label{eqMthetaBorel}
m_{\theta}(B)\leq \left|B\right|^{\frac{s}{\theta}-\frac{2}{p-1}}.
\end{equation}
\end{lem}

\begin{proof}
We follow the same ideas of \cite{BucSeu1}. Let $\eta_{p}=2^{-J_{N_{p}}}$. Let $B$ be a Borel set such that $B\subset K$ with $\left|B\right|<\eta_{p}$. Let $q\geq p+1$ the unique integer such that
\[2^{-J_{N_{q}}}\leq \left|B\right|<2^{-J_{N_{q-1}}}.\]

\medskip
Since for all $V,V'\in F_{q-1}$, $V=\overline{B}(x_{\mathbf{j}},2^{-\theta J_{N_{q-1}}})\cap K$, $V'=\overline{B}(x_{\mathbf{j'}},2^{-\theta J_{N_{q-1}}})\cap K$, we have $\overline{B}(x_{\mathbf{j}},2^{-J_{N_{q-1}}})\cap\overline{B}(x_{\mathbf{j'}},2^{-J_{N_{q-1}}})=\emptyset$ then  $B$ intersect at most $C$ elements of $F_{q-1}$, where $C$ is a constant independent of $p$.

\medskip
Let us distinguish two cases

$\bullet$ $2^{-\theta J_{N_{q-1}}}\leq \left|B\right|<2^{-J_{N_{q-1}}}$: if $B$ dont intersect no one of $F_{q-1}$ then $m_{\theta}(B)=0$. Otherwise, denoting by $V$ any one of $F_{q-1}$ intersecting $B$. Using (\ref{eq1Mtheta}) we have
\begin{eqnarray*}
m_{\theta}(B)&\leq& C m_{\theta}(V)\leq C2^{-sJ_{N_{q-1}}(1-\frac{2}{q-1})}\\
&\leq&C\left|B\right|^{\frac{s}{\theta}(1-\frac{2}{q-1})}\leq \left|B\right|^{\frac{s}{\theta}(1-\frac{2}{p-1})}.
\end{eqnarray*}

$\bullet$ $2^{-J_{N_{q}}}\leq \left|B\right|<2^{-\theta J_{N_{q-1}}}$: Let $V\in F_{q-1}$ that intersect $B$ (if there is no such one then $m_{\theta}(B)=0$). We have proved that for any $U\in F_{q}$, such that $U\subset V$,

\[m_{\theta}(U)=\frac{m_{\theta(V)}}{\Delta_{q}(V)}.\]

By Lemma \ref{lemCardinalFp} we have $\Delta_{q}(V)\geq c'_{1}2^{-s\theta J_{N_{q-1}}}2^{sJ_{N_{q}}}$, then
\begin{equation}\label{eqMthetaUFp}
m_{\theta}(U)\leq \frac{1}{c'_{1}} m_{\theta}(V)2^{s\theta J_{N_{q-1}}-sJ_{N_{q}}}.
\end{equation}
In the other hand, since $B$ is within a ball of side length $C\left|B\right|$, where $C\geq \max\{2,\left|K\right|\}$, the number of elements of $F_{q}$ that intersecting $B$ is less than $c\left|B\right|^{s}2^{sJ_{N_{q}}}$.

Indeed, $B\subset B'$ where $B'$ is a ball of side length  $C\left|B\right|$. If $U=\overline{B}(x_{\mathbf{j}},2^{-\theta J_{N_{q}}})\cap K$ such that $U\cap B\neq\emptyset$, then $K_{\mathbf{j}}\subset B'$ (since $\left|K_{\mathbf{j}}\right|\leq \left|K\right|2^{-J_{N_{q}}}$ and then $K_{\mathbf{j}}\subset B(x_{\mathbf{j}},\left|K\right|2^{-J_{N_{q}}})\subset B'$). Denote by $L$ the set of $U\in F_{p}$ such that $U\cap B=\emptyset$ and $S$ the set of $\mathbf{j}\in I(2^{-J_{N_{q}}})$ such that $K_{\mathbf{j}}\subset B'$, for all such $\mathbf{j}$ we have $K_{\mathbf{j}}\subset B'\cap K$. We have $\sharp L\leq \sharp S$. But
\[\lambda\left(\bigcup_{\mathbf{j}\in S}K_{\alpha}\right)\leq \lambda\left(B'\cap K\right)\leq c\left|B'\right|^{s}\leq c'\left|B\right|^{s}\]
since
\[\lambda\left(\bigcup_{\mathbf{j}\in S}K_{\alpha}\right)=\sum_{\mathbf{j}\in S}\lambda\left(K_{\mathbf{i}}\right)\geq c''\sharp S2^{-sJ_{N_{q}}}\]
we get
\[\sharp L\leq \sharp S\leq c\left|B\right|^{s}2^{sJ_{N_{q}}}.\]

Hence, gathering all the estimation above and the fact that $\left|B\right|^{-\frac{1}{\theta}}>2^{J_{N_{q-1}}}$ we get

\begin{eqnarray*}
m_{\theta}(B)&\leq&\sum_{U\in L}m_{\theta}(U)\\
&\leq& c'\left|B\right|^{s}2^{sJ_{N_{q}}}m_{\theta}(V)2^{s\theta J_{N_{q-1}}+sJ_{N_{q}}}\leq C\left|B\right|^{s}2^{s\theta J_{N_{q-1}}}m_{\theta}(V)\\
&\leq&C\left|B\right|^{s}2^{s\theta J_{N_{q-1}}}2^{-sJ_{N_{q-1}}(1-\frac{2}{p-1})}\leq C\left|B\right|^{s}2^{sJ_{N_{q-1}}(\theta-1+\frac{2}{q-1})}\\
&\leq&C\left|B\right|^{s}\left|B\right|^{-\frac{1}{\theta}s(\theta-1+\frac{2}{q-1})}\leq C\left|B\right|^{\frac{s}{\theta}(1-\frac{2}{p-1})}.
\end{eqnarray*}
\end{proof}

Already we have all the ingredients to finish the proof of Theorem \ref{maintheorem} with the same way as in \cite{BucSeu1} by considering the sets
\[\widetilde{E}_{\mu}(h)=\left\{x\in K:\;h_{\mu}(x)\leq h\right\}=\bigcup_{h'\leq h}E_{\mu}(h).\]
For safe completeness we recover their idea.
\begin{prop}\label{propSpectre}
For any $h\in]0,s]$, $d_{\mu}(h)=h$.
\end{prop}

\begin{proof}
Let $h\in]0,s]$, and $\theta=\frac{s}{\theta}$.

A standard result claim that for all $h\geq 0$,
\begin{equation}\label{eqdimEmuhtilde}
\dimh \widetilde{E}_{\mu}(h)\leq \min\left\{h,d\right\}.
\end{equation}

Then, by Proposition \ref{propDimHausdorffLambdaTheta}, $\Lambda_{\theta}\subset \widetilde{E}_{\mu}(h)$. Let us write
\[\Lambda_{\theta}=\left(\Lambda_{\theta}\cap E_{\mu}(h)\right)\bigcup \left(\bigcup_{n\geq1}\Lambda_{\theta}\cap\widetilde{E}_{\mu}(h-\frac{1}{n})\right).\]
Now, consider the measure $m_{\theta}$ provided by Theorem \ref{thMtheta} which is supported by the Cantor set $\mathcal{C}_{\theta}\subset \Lambda_{\theta}$. We have then
\[m_{\theta}(\Lambda_{\theta})\geq m_{\theta}(\Lambda_{\theta})>0.\]
By (\ref{eqdimEmuhtilde}), for any $n\geq 1$, $\dimh\left(\Lambda_{\theta}\cap\widetilde{E}_{\mu}(h-\frac{1}{n})\right)\leq h-\frac{1}{n}<h$. Since $\dimh m_{\theta}\geq \frac{s}{\theta}=h$, then by property (\ref{propertyDimMesure}) we deduce that $m_{\theta}\left(\Lambda_{\theta}\cap\widetilde{E}_{\mu}(h-\frac{1}{n})\right)=0$.

\medskip
Then, we get $m_{\theta}(\Lambda_{\theta})=m_{\theta}\left(\Lambda_{\theta}\cap E_{\mu}(h)\right)>0$. Hence, again by property (\ref{propertyDimMesure})
\[\dimh E_{\mu}(h)\geq \dimh \Lambda_{\theta}\cap E_{\mu}(h)\geq \frac{s}{\theta}=h.\]
Finally, the upper bound results from the inclusion $E_{\mu}(h)\subset \widetilde{E}_{\mu}(h)$. This finish the proof of the Proposition \ref{propSpectre}.
\end{proof}

\bigskip
It remains the case $h=0$.

\medskip
Let $x_{0}\in K$ any fixed point. By Theorem \ref{th1}, there exists a $G_{\delta}$ set $\Omega(x_{0})$ of $\mathcal{M}(K)$ such that for all $\mu\in \Omega(x_{0})$, $h_{\mu}(x_{0})=0$. Thus, for all $\mu\in \Omega(x_{0})$, $x_{0}\in E_{\mu}(0)$. Hence, for all $\mu\in \Omega(x_{0})$, $E_{\mu}(0)\neq\emptyset$. Thus, for all $\mu\in \Omega(x_{0})$, $\dimh E_{\mu}(0)\geq 0 $. As already we have the upper bound, then for all $\mu\in \Omega(x_{0})$, $\dimh E_{\mu}(0)=0$.

\bigskip
Consider $\Omega=\Omega'\cap\Omega''\cap\Omega(x_{0})$. $\Omega$ is a $G_{\delta}$ set of $\mathcal{M}(K)$ and for all $\mu in\Omega$, $\mu$ satisfy all the points of Theorem \ref{maintheorem}.

\medskip
It remains for us to show that any $\mu\in \Omega$ satisfies the multifractal formalism.

\medskip
Let $\mu\in \mathcal{M}(K)$. We denote by $N_{j}(K)$, the number of cubes of $\mathcal{G}_{j}$ that intersect $K$. By the concavity of $t\mapsto t^{q}$, for $q\in[0,1]$, we have for any $q\in[0,1]$
\begin{eqnarray*}
\sum_{Q\in\mathcal{G}_{j},\mu(Q)\neq 0}\mu(Q)^{q}&=&\sum_{Q\in\mathcal{G}_{j},K\cap Q\neq \emptyset}\mu(Q)^{q}\\
&\leq& N_{j}(K)\left(\frac{1}{N_{j}(K)}\sum_{Q\in\mathcal{G}_{j},K\cap Q\neq \emptyset}\mu(Q)\right)^{q}=N_{j}(K)^{1-q}.
\end{eqnarray*}
Thus, for all $q\in[0,1]$
\begin{eqnarray*}
\tau_{\mu}(q)=\liminf_{j\to+\infty}-\frac{\log\sum_{Q\in\mathcal{G}_{j},\mu(Q)\neq0 }\mu(Q)^{q}}{j\log2}&\geq&(q-1)\limsup_{j\to\infty}\frac{\log N_{j}(K)}{j\log2}\\
&=&(q-1)\dimbsup(K)=(q-1)s.
\end{eqnarray*}

\medskip
Let $\mu\in\Omega$. From (\ref{majholdmu}) it follows that for all $h\in[0,s]$,
\begin{equation}\label{eqEncadLegendreSpectre}
d_{\mu}(h)=h\leq\left(\tau_{\mu}\right)^{*}(h)\leq \inf_{q\in[0,1]}(qh-\tau_{\mu}(q)).
\end{equation}
Hence, for all $h\in[0,s]$, for all $q\in[0,1]$,  $\tau_{\mu}(q)\leq (q-1)h$. In particular, for $h=s$, for all $q\in[0,1]$, $\tau_{\mu}(q)\leq (q-1)s$. As we already have the lower bound, we conclude that for all $q\in[0,1]$, $\tau_{\mu}(q)=(q-1)s$. Then, for all $h\in[0,s]$,
\[\inf_{q\in[0,1]}(qh-\tau_{\mu}(q))=h=d_{\mu}(h).\]
Thus, the inequalities of (\ref{eqEncadLegendreSpectre}) turn to be equalities. Hence, for all $h\in[0,s]$,
\[d_{\mu}(h)=h=\left(\tau_{\mu}\right)^{*}(h).\]

\end{proof}

 \end{document}